\documentclass[12pt]{amsart}
\usepackage[utf8]{inputenc}
\usepackage{amsfonts}
\usepackage{amssymb}
\usepackage{amsmath}
\usepackage[english]{babel}
\usepackage{amsthm}
\theoremstyle{plain}
\newtheorem{theorem}{Theorem}[section]

\newtheorem{lemma}[theorem]{Lemma}

\theoremstyle{definition}
\newtheorem{definition}[theorem]{Definition}
\newtheorem{remark}[theorem]{Remark}
\usepackage{geometry}
\numberwithin{equation}{section} 
\usepackage{caption}
\usepackage{graphicx}
\usepackage{diagbox}
\usepackage{placeins}
\usepackage{float}

\usepackage{lipsum}

\newcommand\blfootnote[1]{%
  \begingroup
  \renewcommand\thefootnote{}\footnote{#1}%
  \addtocounter{footnote}{-1}%
  \endgroup
}

\usepackage[titletoc,toc,title]{appendix}

\usepackage[colorlinks, citecolor=blue,pagebackref,hypertexnames=false]{hyperref}

\newcounter{comcount}

\begin{document}

\title{Some Remarks on the Riesz and reverse Riesz transforms on Broken Line} 
\date{}
\author{DANGYANG HE}  
\address{Department of Mathematics and Statistics, Macquarie University}
\email{dangyang.he@students.mq.edu.au, hedangyang@gmail.com}

\begin{abstract}
In this note, we study both the Riesz and reverse Riesz transforms on broken line. This model can be described by
\begin{equation*}
    \Tilde{\mathbb{R}}=(-\infty,-1]\cup [1,\infty)
\end{equation*}
equipped with the measure
\begin{align*}
    d\mu = \begin{cases}
        |r|^{d_1-1}dr, & r\le -1,\\
        r^{d_2-1}dr, & r\ge 1,
    \end{cases}
\end{align*}
where $d_1,d_2>1$. For the Riesz transform, we show that the range of its $L^p\textit{-}$boundedness depends solely on the smaller “dimension", $d_*:=d_1 \wedge d_2$. Furthermore, we establish a Lorentz-type estimate at the endpoint. In our subsequent investigation, we consider the reverse Riesz inequality by rigorously verifying the $L^p\textit{-}$lower bound for the Riesz transform for almost every $p\in (1,\infty)$. Notably, unlike most previous studies, we do not assume the doubling condition or the Poincaré inequality. Our approach is based on careful estimates of the Riesz kernel and a method known as harmonic annihilation.
\end{abstract}
\maketitle

\tableofcontents

\blfootnote{$\textit{2020 Mathematics Subject Classification.}$ 42B20, 47F05}

\blfootnote{$\textit{Keywords and Phrases.}$ Riesz transform, reverse Riesz transform, endpoint estimates.}

\section{Introduction}
The Riesz transform, named after Marcel Riesz, is a fundamental operator in harmonic analysis and partial differential equations. The Riesz transform is instrumental in several areas, including characterizing function spaces such as Hardy and Sobolev spaces, establishing boundedness properties on $L^p$ spaces for $p\in (1,\infty)$, and solving elliptic partial differential equations. Its connections to the theory of singular integrals make it a critical tool in modern analysis, influencing a wide array of applications in both theoretical and applied mathematics.

On the classical Euclidean spaces $\mathbb{R}^n$, Riesz \cite{R} proved that the one-dimensional Riesz transform (Hilbert transform) is bounded on $L^p$ for all $p\in (1,\infty)$ by using complex analysis and an iteration argument. Extending this result to the multidimensional case was fraught with challenges until the emergence of the Calderón–Zygmund decomposition \cite{CZ}. In \cite{S}, several process about extending the results from standard Laplace operator to the setting of Laplace-Beltrami operator on the complete Riemannian manifolds have been established. However, in the case of Riesz transform, the question turns out to be extremely difficult. Particularly, if the manifold is volume doubling, the approaches between situations $p<2$ and $p>2$ are completely different, see \cite{CD} for $p<2$ and \cite{ACDH} for $p>2$, see also \cite{C,AC,B2,B,CD2,L} and references therein. Regarding the manifolds where the doubling condition fails, we refer readers to, for example \cite{HNS,HS,H,N}.

In this note, we build on the results of \cite{HS1D} and \cite{N,nix2019resolvent} to further investigate the Riesz transform in a one-dimensional setting. Specifically, we consider the space
\begin{align*}
    \Tilde{\mathbb{R}}= (-\infty,-1]\cup [1,\infty),
\end{align*}
endowed with the measure
\begin{equation}\label{eq_measure}
    d\mu(r) = \begin{cases}
        |r|^{d_1-1}dr, & r\le -1,\\
        r^{d_2-1}dr, & r\ge 1,
    \end{cases}
\end{equation}
where $dr$ is the usual Lebesgue measure and $d_1,d_2>1$. We define the operator
\begin{align}\label{eq_laplace}
    \Delta = \nabla^*\nabla
\end{align}
to be the Laplacian on the broken line $\Tilde{\mathbb{R}}$, where $\nabla f=f'$ is the usual derivative operator and $\nabla^*$ is its formal adjoint with respect to the measure $d\mu$. The Riesz transform is then given by $\nabla \Delta^{-1/2}$. Throughout the remainder of this note, and without loss of generality, we will assume $1<d_1<d_2$. Our motivation for studying this model stems from \cite{HS1D,GH1,GH2,Davies}.

Roughly speaking, this model captures the “radial behavior” of the Riesz transform on manifolds with ends. For example, if $d\ge 2$ is an integer, then $\nabla \Delta^{-1/2}$ models the radial part of the Riesz transform on two copies of $\mathbb{R}^d$, as considered in \cite{CCH}. Moreover, if $d_1\ne d_2$, then $\nabla \Delta^{-1/2}$ essentially corresponds to the spherical symmetric part of the Riesz transform examined in \cite{HNS,HS,H}. Nevertheless, here we consider all real $d_1,d_2>1$. 

Let $1\le p\le \infty$ and let $p'$ be the conjugate exponent, i.e. $1/p+1/p'=1$. From \cite{HS1D}, we know that if both “ends” of $\Tilde{\mathbb{R}}$ share the same “dimension,” the following theorem holds. 

\begin{theorem}{\cite[Theorem 5.6]{HS1D}}\label{thm_HS}
Let $d_1=d_2=d>1$ and $\Delta$ be as in \eqref{eq_laplace}. The Riesz transform, $\nabla \Delta^{-1/2}$, is bounded on $L^p(\Tilde{\mathbb{R}},d\mu)$ if and only if\\
$(\romannumeral1)$ $1<p<d$ for $d>2$.\\
$(\romannumeral2)$ $1<p\le 2$ for $d=2$.\\
$(\romannumeral3)$ $1<p<d'$ for $1<d<2$.
\end{theorem}

Another important result appears in Nix’s PhD thesis \cite[Theorem 2.1]{N}, where the author examines the scenario in which one “end” has the critical dimension two, while the other “end” has dimension strictly greater than two. In fact, although \cite{N} specifically addresses the case $d_1=2$ and $d_2 \ge 3$, there is no fundamental obstruction to extending the main argument to all $d_2>2$.

\begin{theorem}{\cite[Theorem 2.1]{N}}\label{thm_Nix}
 Let $2=d_1<d_2$ and $\Delta$ be as in \eqref{eq_laplace}. Then, $\nabla \Delta^{-1/2}$ is bounded on $L^p(\Tilde{\mathbb{R}},d\mu)$ if and only if $1<p\le 2$.
\end{theorem}

Generally, one considers the Dirichlet Laplacian on the exterior of a convex, bounded obstacle in $\mathbb{R}^d$ with $d\ge 2$. According to \cite{KVZ} and \cite{JL}, the associated Riesz transform is $L^p\textit{-}$bounded precisely when $1<p<d$ for $d\ge 3$ and $1<p\le 2$ for $d=2$.
 
In the first part of this note, we extend Theorems~\ref{thm_HS} and~\ref{thm_Nix} by examining the following four cases:
\begin{align*}
    &\bullet 1<d_1<d_2<2, \quad
&&\bullet 1<d_1<d_2=2,\\
&\bullet 1<d_1<2<d_2, \quad
&&\bullet 2<d_1<d_2.
\end{align*}
Before presenting the main results of this note, let us first recall some notation from the theory of Lorentz spaces. For a measurable function $f$, let $d_f$ be its distribution function, defined by
\begin{equation*}
    d_f(s) = \mu(\{x; |f(x)| >s\}).
\end{equation*}
We denote by $f^*$, the decreasing rearrangement of $f$, given by
\begin{equation*}
    f^*(t) = \inf\{s>0; d_f(s)\le t\}.
\end{equation*}

\begin{definition}
Let $f$ be a measurable function defined on the measure space $(X,\mu)$. For $0<p,q\le \infty$, define
\begin{equation*} 
     \|f\|_{(p,q)}=  
    \begin{cases}
    \left( \int_0^{\infty} \left(t^{1/p} f^*(t)\right)^q \frac{dt}{t}   \right)^{\frac{1}{q}}, & q<\infty,\\
     \sup_{t>0} t^{1/p}f^*(t), & q=\infty.
    \end{cases}
\end{equation*}
Then, we say $f$ is in the Lorentz space $L^{p,q}(X,\mu)$ if $\|f\|_{(p,q)}<\infty$.
\end{definition}

It is well-known that $L^{p,p}=L^p$ and that $L^{p,\infty}$ coincides with the weak $L^p$ space. For a more in-depth discussion of Lorentz spaces, we refer readers to \cite{G}. Throughout this note, we use the notations $a\vee b = \max(a,b)$ and $a\wedge b=\min(a,b)$. Our first result can be stated as follows.

\begin{theorem}\label{thm_main}
 Let $d_1,d_2>1$ and $\Delta$ be as in \eqref{eq_laplace}. The Riesz transform $\nabla \Delta^{-1/2}$ is bounded on $L^p(\Tilde{\mathbb{R}},d\mu)$ if and only if
 \begin{equation*}
     1<p< p_0:= d_* \vee d_*'\quad \text{or} \quad p=2,
 \end{equation*}
 where $d_*=d_1\wedge d_2$.

 In addition, $\nabla \Delta^{-1/2}$ is of restricted weak type $(p_0,p_0)$ i.e. bounded from $L^{p_0,1}\to L^{p_0,\infty}$.
\end{theorem}

As observed in the result above, the range of $L^p\textit{-}$boundedness of the Riesz transform depends only on the smaller “dimension” $d_*$. This outcome is consistent with earlier findings in \cite{HS, HS1D, HNS, H}, and it extends those results by accounting for all remaining possible dimensional configurations on each end.

We next turn to the study of a reverse Riesz inequality. Specifically, we aim to establish an inequality of the following form:
\begin{align*}
    \|\Delta^{1/2}f\|_p \le C \|\nabla f\|_p,
\end{align*}
for some range of $p\in (1,\infty)$. It is well-known that on a complete Riemannian manifold, the $L^p\textit{-}$boundedness of the Riesz transform directly implies the $L^{p'}\textit{-}$boundedness of the reverse Riesz transform (duality property), but the converse does not generally hold; see \cite{CD}. In contrast to the classical Euclidean setting, there exist certain classes of manifolds for which the Riesz transform is bounded on only a finite subset of $(1,\infty)$; see \cite{HS,HNS,L,GH1,GH2}. Consequently, by duality, one may naturally conjecture that under the framework of \cite{HS,HNS,L}, the range of boundedness for the reverse inequality would also exclude a portion of $(1,\infty)$. However, recent works \cite{H2} demonstrate that this presumed dual equivalence between the Riesz and reverse Riesz transforms is false. In fact, earlier literature already contains strong indications of such a discrepancy. For example, \cite{AC} shows that when a doubling condition and an $L^q\textit{-}$Poincaré inequality both hold, the reverse Riesz transform is $L^p\textit{-}$bounded for all $p>q$. Under the same assumptions, however, one can only prove that the Riesz transform is bounded for $1<p<2+\epsilon$ for some $\epsilon>0$. 

Define subspace
\begin{align*}
    S_0(\Tilde{\mathbb{R}}) = \left\{f\in C_c^\infty(\Tilde{\mathbb{R}}); f(-1) = f(1) = 0\right\}.
\end{align*}
In the second part of this note, we adapt the method introduced in \cite{H2} to complete the analysis of the reverse Riesz transform on the broken line, building on the following partial results in \cite{H2}:

\begin{theorem}\cite[Theorem~5.1]{H2}
Let $d_1=d_2>1$and $\Delta$ be as in \eqref{eq_laplace} and $\Delta$ be as in \eqref{eq_laplace}. The following reverse Riesz inequality holds
\begin{align*}
    \|\Delta^{1/2}f\|_p \le C \|\nabla f\|_p,\quad \forall f\in S_0(\Tilde{\mathbb{R}})
\end{align*}
for all
\begin{align*}
    p\in \begin{cases}
        (1,d) \cup (d,\infty), & 1<d<2,\\
        (1,\infty), & d\ge 2.
    \end{cases}
\end{align*}
\end{theorem}

Our next result can be formulated as follows.

\begin{theorem}\label{thm_RR_main}
Let $d_1,d_2>1$ and $\Delta$ be as in \eqref{eq_laplace}. The following reverse Riesz inequality holds
\begin{align*}
    \|\Delta^{1/2}f\|_p \le C \|\nabla f\|_p,\quad \forall f\in S_0(\Tilde{\mathbb{R}})
\end{align*}
for all 
\begin{align}\label{range}
    p\in \begin{cases}
        (1,d_*) \cup (d_*,\infty), & \textit{if}\quad 1<d_*<2,\\
        (1,\infty), & \textit{if}\quad d_*\ge 2,
    \end{cases}
\end{align}
where $d_* = d_1 \wedge d_2$.
\end{theorem}

\section{Riesz Transform on Broken Line}\label{sec_2}

Throughout the paper, we use notations $A\lesssim B$, $A\gtrsim B$ and $A\simeq B$ to denote $A\le cB$, $A\ge cB$ and $cB \le A \le c^{-1} B$ respectively for some constant $c>0$. 
\medskip

Recall the broken line $\Tilde{\mathbb{R}} = (-\infty,-1]\cup [1,\infty)$ and the measure defined in \eqref{eq_measure}. We say that $f\in C^1(\Tilde{\mathbb{R}})$ if $f$ is continuous on $(-\infty, -1]\cup[1,\infty)$, with the additional conditions $f(1) = f(-1)$ and $f'(-1)=f'(1)$. We then define the Sobolev space 
\begin{equation*}
    H^1(\Tilde{\mathbb{R}},d\mu) = \left\{f:\Tilde{\mathbb{R}} \mapsto \mathbb{C}: \int_{\Tilde{\mathbb{R}}}|f|^2+|f'|^2d\mu <\infty \right\}.
\end{equation*}
For $f,g\in H^1\cap C^0$, we define quadratic form:
\begin{equation}\label{eq_quadratic form}
    \mathcal{Q}(f,g) = \int_{-\infty}^{-1}f'(r)g'(r)|r|^{d_1-1}dr+ \int_1^\infty f'(r)g'(r)r^{d_2-1}dr.
\end{equation}
By Friedrichs’ extension, there is a unique self-adjoint operator $\Delta$ associated with \eqref{eq_quadratic form}, which takes the form
\begin{equation*}
    \Delta f = 
    \begin{cases}
    -f{''}(r) - \frac{d_1-1}{r}f'(r), & r\le -1\\
    -f{''}(r) - \frac{d_2-1}{r}f'(r), & r\ge 1.
    \end{cases}
\end{equation*}
We say $\Delta$ is the Laplacian on $\Tilde{\mathbb{R}}$, with domain $H^2\cap C^1$, where
\begin{equation*}
    H^2(\Tilde{\mathbb{R}}) = \left\{ f\in H^1(\Tilde{\mathbb{R}}): \Delta f\in L^2(\Tilde{\mathbb{R}},d\mu)\right\}.
\end{equation*}
Let $\nabla$ be the usual derivative operator i.e. $\nabla f = f'$. The Riesz transform is then defined by $\nabla \Delta^{-1/2}$. Following the approach from \cite{HS1D}, we stick to a resolvent based formula
\begin{gather*}
    \nabla \Delta^{-1/2} = \frac{2}{\pi} \int_0^\infty \nabla(\Delta+\lambda^2)^{-1}d\lambda.
\end{gather*}
From now on, we omit the constant $2/\pi$ in subsequent expressions since it plays no role in the argument.

Consider ordinary differential equation $(i=1,2)$:
\begin{equation}\label{eq_ODE}
    f^{''} + \frac{d_i-1}{r}f' = \lambda^2 f.
\end{equation}
By \cite{HS1D}, the general solution of \eqref{eq_ODE} is given by a linear combination of the functions $l_i(\lambda r)$ and $k_i(\lambda r)$, defined by 
\begin{equation*}
    l_i(r) = r^{1-d_i/2}I_{d_i/2-1}(r), \quad k_i(r) = r^{1-d_i/2}K_{|d_i/2-1|}(r),
\end{equation*}
where $I$ and $K$ denote the modified Bessel functions, see \cite[p.374]{AS}. Following \cite[Chapter 2]{N,nix2019resolvent} (also see \cite[Section 4]{HS1D}), one can compute the resolvent kernel explicitly:

\textbf{For $y\ge 1$}
\begin{equation}\label{eq_y>1}
    K_{(\Delta+\lambda^2)^{-1}}(x,y) = 
    \begin{cases}
    A(\lambda) k_1(\lambda|x|) k_2(\lambda|y|), & x\le -1,\\
    B(\lambda) k_2(\lambda|y|) k_2(\lambda|x|) + v_2 \lambda^{d_2-2}k_2(\lambda|y|)l_2(\lambda|x|), & 1\le x\le y,\\
    B(\lambda) k_2(\lambda|y|) k_2(\lambda|x|) + v_2 \lambda^{d_2-2}l_2(\lambda|y|) k_2(\lambda|x|), & x\ge y.
    \end{cases}
\end{equation}

\textbf{For $y\le -1$}
\begin{equation}\label{eq_y<1}
    K_{(\Delta+\lambda^2)^{-1}}(x,y) = 
    \begin{cases}
    A(\lambda) k_2(\lambda|x|) k_1(\lambda|y|), & x\ge 1,\\
    C(\lambda) k_1(\lambda|y|) k_1(\lambda|x|) + v_1 \lambda^{d_1-2}k_1(\lambda|y|)l_1(\lambda|x|), & y\le x\le -1,\\
    C(\lambda) k_1(\lambda|y|) k_1(\lambda|x|) + v_1 \lambda^{d_1-2}l_1(\lambda|y|) k_1(\lambda|x|), & x\le y,
    \end{cases}
\end{equation}
where $v_i$ are some constants only depending on $d_i$ ($i=1,2$), and
\begin{gather*}
    A(\lambda) = \frac{-1}{\lambda [k_1(\lambda) k_2(\lambda)]'},\quad B(\lambda) = \frac{-v_2 \lambda^{d_2-2}[k_1(\lambda) l_2(\lambda)]'}{[k_1(\lambda) k_2(\lambda)]'},\quad
    C(\lambda) = \frac{-v_1 \lambda^{d_1-2}[k_2(\lambda) l_1(\lambda)]'}{[k_1(\lambda) k_2(\lambda)]'}.
\end{gather*}
We classify the kernel described in \eqref{eq_y>1} and \eqref{eq_y<1} into two groups:

$\bullet$ The \emph{"kk"} part, denoted by $(\Delta+\lambda^2)^{-1}_{kk}$, which consists of all terms of the form $k(\cdot)k(\cdot)$.

$\bullet$ The \emph{"kl"} part, denoted by $(\Delta+\lambda^2)^{-1}_{kl}$, which comprises the remaining terms, i.e. those of the form $k(\cdot)l(\cdot)$. Apparently
\begin{equation*}
(\Delta+\lambda^2)^{-1} = (\Delta+\lambda^2)^{-1}_{kk} + (\Delta+\lambda^2)^{-1}_{kl}. 
\end{equation*}

Next, from \cite[p.,374]{AS}, we recall the following asymptotic behaviors for $k_i$ and $l_i$.

\textbf{If $1<d_i<2$}:
\begin{align*}
    &k_i(r)\simeq 
    \begin{cases}
    1, & r<1,\\
    r^{(1-d_i)/2}e^{-r}, & r\ge 1,
    \end{cases}
    \quad
    &&l_i(r)\simeq 
    \begin{cases}
    1, & r<1,\\
    r^{(1-d_i)/2}e^r, & r\ge 1,
    \end{cases}
    \\
    &k_i'(r)\simeq 
    \begin{cases}
    -r^{1-d_i}, & r<1,\\
    -r^{(1-d_i)/2}e^{-r}, & r\ge 1,
    \end{cases}
    \quad
    &&l_i'(r)\simeq 
    \begin{cases}
    r, & r<1,\\
    r^{(1-d_i)/2}e^r, & r\ge 1.
    \end{cases}
\end{align*}

\textbf{If $d_i=2$}:
\begin{align*}
    &k_i(r)\simeq 
    \begin{cases}
    1-\log \lambda, & r<1,\\
    r^{-1/2}e^{-r}, & r\ge 1,
    \end{cases}
    \quad
    &&l_i(r)\simeq 
    \begin{cases}
    1, & r<1,\\
    r^{-1/2}e^r, & r\ge 1,
    \end{cases}
    \\
    &k_i'(r)\simeq 
    \begin{cases}
    -r^{-1}, & r<1,\\
    -r^{-1/2}e^{-r}, & r\ge 1,
    \end{cases}
    \quad
    &&l_i'(r)\simeq 
    \begin{cases}
    r, & r<1,\\
    r^{-1/2}e^r, & r\ge 1.
    \end{cases}
\end{align*}

\textbf{If $d_i>2$}:
\begin{align*}
    &k_i(r)\simeq 
    \begin{cases}
    r^{2-d_i}, & r<1,\\
    r^{(1-d_i)/2}e^{-r}, & r\ge 1,
    \end{cases}
    \quad
    &&l_i(r)\simeq 
    \begin{cases}
    1, & r<1,\\
    r^{(1-d_i)/2}e^r, & r\ge 1,
    \end{cases}
    \\
    &k_i'(r)\simeq 
    \begin{cases}
    -r^{1-d_i}, & r<1,\\
    -r^{(1-d_i)/2}e^{-r}, & r\ge 1,
    \end{cases}
    \quad
    &&l_i'(r)\simeq 
    \begin{cases}
    r, & r<1,\\
    r^{(1-d_i)/2}e^r, & r\ge 1.
    \end{cases}
\end{align*}

Using the asymptotic formulas above, one can directly estimate the coefficients $A,B$, and $C$ that appear in \eqref{eq_y<1} and \eqref{eq_y>1}.

\textbf{If $1<d_1<d_2<2$}.
\begin{gather*}
    A\simeq 
    \begin{cases}
    \lambda^{d_2-2}, & \lambda<1\\
    \lambda^{\frac{d_1+d_2}{2}-2}e^{2\lambda}, & \lambda \ge 1
    \end{cases}\quad
    |B|\lesssim 
    \begin{cases}
    \lambda^{2d_2-d_1-2}, & \lambda<1\\
    \lambda^{d_2-2}e^{2\lambda}, & \lambda \ge 1
    \end{cases}\quad
    |C|\lesssim
    \begin{cases}
    \lambda^{d_1-2}, & \lambda<1\\
    \lambda^{d_1-2}e^{2\lambda}, & \lambda \ge 1
    \end{cases}.
\end{gather*}\

\textbf{If $1<d_1<d_2=2$}. 
\begin{gather*}
    A\simeq \begin{cases}
        1, & \lambda<1\\
        \lambda^{\frac{d_1}{2}-1}e^{2\lambda}, & \lambda\ge 1
    \end{cases}\quad
    |B|\lesssim \begin{cases}
        \lambda^{2-d_1}, & \lambda<1\\
        e^{2\lambda}, & \lambda \ge 1
    \end{cases}\quad
    |C|\lesssim \begin{cases}
        \lambda^{d_1-2}, & \lambda<1\\
        \lambda^{d_1-2}e^{2\lambda}, & \lambda \ge 1
    \end{cases} .
\end{gather*}

\textbf{If $1<d_1<2<d_2$}.
\begin{gather*}
    A\simeq 
    \begin{cases}
    \lambda^{d_2-2}, & \lambda<1\\
    \lambda^{\frac{d_1+d_2}{2}-2}e^{2\lambda}, & \lambda \ge 1
    \end{cases}\quad
    |B|\lesssim 
    \begin{cases}
    \lambda^{2d_2-d_1-2}, & \lambda<1\\
    \lambda^{d_2-2}e^{2\lambda}, & \lambda \ge 1
    \end{cases}\quad
    |C|\lesssim
    \begin{cases}
    \lambda^{d_1-2}, & \lambda<1\\
    \lambda^{d_1-2}e^{2\lambda}, & \lambda \ge 1.
    \end{cases}
\end{gather*}

\textbf{If $2<d_1<d_2$}. 
\begin{gather*}
    A\simeq \begin{cases}
        \lambda^{d_1+d_2-4}, & \lambda <1\\
        \lambda^{\frac{d_1+d_2}{2}-2}e^{2\lambda}, & \lambda \ge 1
    \end{cases}\quad
    |B|\lesssim \begin{cases}
        \lambda^{2d_2-4}, & \lambda<1\\
        \lambda^{d_2-2}e^{2\lambda}, & \lambda \ge 1
    \end{cases}\quad
    |C|\lesssim \begin{cases}
        \lambda^{2d_1-4}, & \lambda<1\\
        \lambda^{d_1-2}e^{2\lambda}, & \lambda \ge 1.
    \end{cases}
\end{gather*}

\begin{remark}
Note that, by \cite[Lemma 3.1]{HS1D} (although different notation is used there), the coefficient $A$ is always positive. The signs of $B$ and $C$, however, are not determined. Here, we only provide upper bounds for $B$ and $C$. More precise estimates can be obtained via the series expansion of the modified Bessel functions, revealing that $B$ and $C$ can, in fact, assume a fixed sign when $\lambda$ is sufficiently small; see \cite{N,nix2019resolvent} for details.
\end{remark}

\subsection{Kernel Estimates}
Recall that by \eqref{eq_y>1}, \eqref{eq_y<1}, the Riesz transform can be written as
\begin{equation}\label{eq_riesz}
    \nabla \Delta^{-1/2}= \int_0^\infty \nabla(\Delta+\lambda^2)^{-1}_{kk}d\lambda + \int_0^\infty \nabla(\Delta+\lambda^2)^{-1}_{kl}d\lambda:= R_{kk}+R_{kl}.
\end{equation}
In this subsection, we give estimates for the kernel of $R_{kk}$. We define "quadrants"
\begin{align*}
    &Q_1 = \{x,y\ge 1\} \quad &&Q_2=\{x\le -1, y\ge 1\}\\
    &Q_3 = \{x,y\le -1\} \quad &&Q_4=\{x\ge 1, y\le -1\}.
\end{align*}
From \eqref{eq_y>1} and \eqref{eq_y<1}, we compute $R_{kk}(x,y)$ directly

\textbf{For} $(x,y)\in Q_1$,
\begin{equation}\label{Q1}
R_{kk}(x,y) = \int_0^{1/(|x|\wedge |y|)}+\int_{1/(|x|\wedge |y|)}^\infty \lambda B(\lambda) k_2(\lambda|y|) k_2'(\lambda|x|)d\lambda:= I_1 + I_2.
\end{equation}

\textbf{For} $(x,y)\in Q_2$,
\begin{equation}\label{Q2}
R_{kk}(x,y)=\int_0^{1/(|x|\wedge |y|)}+\int_{1/(|x|\wedge |y|)}^\infty \lambda A(\lambda) k_1'(\lambda|x|) k_2(\lambda|y|)d\lambda:=II_1+II_2.
\end{equation}

\textbf{For} $(x,y)\in Q_3$,
\begin{equation}\label{Q3}
R_{kk}(x,y)=\int_0^{1/(|x|\wedge |y|)}+\int_{1/(|x|\wedge |y|)}^\infty \lambda C(\lambda) k_1(\lambda|y|) k_1'(\lambda|x|)d\lambda:=III_1+III_2.
\end{equation}

\textbf{For} $(x,y)\in Q_4$,
\begin{equation}\label{Q4}
R_{kk}(x,y)=\int_0^{1/(|x|\wedge |y|)}+\int_{1/(|x|\wedge |y|)}^\infty \lambda A(\lambda) k_2'(\lambda|x|) k_1(\lambda|y|)d\lambda:=IV_1+IV_2.
\end{equation}
Set $T_L = I_1+II_1+III_1+IV_1$ and $T_H=I_2+II_2+III_2+IV_2$. Then obviously, 
\begin{equation}\label{riesz decomposition}
    \nabla \Delta^{-1/2} = T_L+T_H+R_{kl},
\end{equation}
and we estimate the kernel of $T_L$ and $T_H$ separately. 

For the low energy component $T_L$, we introduce a function $F$, which may represent $A$, $B$, or $C$, depending on the quadrant under consideration. Relying on the asymptotics derived in Section~\ref{sec_2}, we assume
$|F(\lambda)|\lesssim \lambda^{\gamma_1}$ if $\lambda \le 1$ and $|F(\lambda)|\lesssim \lambda^{\gamma_2} e^{2\lambda}$ if $\lambda \ge 1$ where $\gamma_1,\gamma_2 \in \mathbb{R}$. It is the enough to estimate the following integral:
\begin{equation}\label{eq_low}
    \int_0^{1/(|x|\wedge |y|)} \lambda F(\lambda) k'_{i}(\lambda|x|) k_j(\lambda|y|) d\lambda \quad i,j\in \{1,2\}.
\end{equation}
Now, if $|x|\ge |y|$ (i.e. $1/(|x|\wedge |y|) = |y|^{-1}$), \eqref{eq_low} can be split into $\int_0^{|x|^{-1}} + \int_{|x|^{-1}}^{|y|^{-1}}$ and the first integral is bounded by
\begin{gather}
    \int_0^{|x|^{-1}} \lambda^{\gamma_1+1} (\lambda |x|)^{1-d_i} (\lambda|y|)^{(2-d_j)\wedge 0} d\lambda = |x|^{1-d_i}|y|^{(2-d_j)\wedge 0} \int_0^{|x|^{-1}} \lambda^{\gamma_1+2-d_i+(2-d_j)\wedge 0} d\lambda\\
    = \begin{cases}
        |x|^{1-d_i} \int_0^{|x|^{-1}}\lambda^{\gamma_1+2-d_i}d\lambda = |x|^{-\gamma_1-2}, & \textit{if}\quad j=1,\\ \nonumber
        |x|^{1-d_i} |y|^{2-d_j} \int_0^{|x|^{-1}}\lambda^{\gamma_1+4-d_i-d_j} d\lambda = |x|^{-\gamma_1+d_j-4}|y|^{2-d_j}, & \textit{if}\quad j=2,
    \end{cases}
\end{gather}
where we use the small variable asymptotics for both $k'_i(\lambda|x|)$ and $k_j(\lambda|y|)$.

Next, the second integral can be estimated by
\begin{gather}
|x|^{\frac{1-d_i}{2}} |y|^{(2-d_j)\wedge 0} \int_{|x|^{-1}}^{|y|^{-1}} \lambda^{\frac{2\gamma_1+3-d_i}{2}+(2-d_j)\wedge 0} e^{-\lambda|x|} d\lambda\\ \nonumber
    \lesssim \begin{cases}
        |x|^{-\gamma_1-2}, & \textit{if}\quad j=1,\\
        |x|^{-\gamma_1-4+d_j}|y|^{2-d_j}, & \textit{if}\quad j=2,
    \end{cases}
\end{gather}
where we use small variable asymptotics for $k_j(\lambda|y|)$ and large variable asymptotics for $k'_i(\lambda|x|)$.

While if $|x|\le |y|$ (i.e. $1/(|x|\wedge |y|)=|x|^{-1}$), we break up integral $\int_0^{|x|^{-1}}$ into $\int_0^{|y|^{-1}} + \int_{|y|^{-1}}^{|x|^{-1}}$. We then use small variable estimates for both $k'_i(\lambda|x|)$ and $k_j(\lambda|y|)$ in the first integral and the second we apply large variable asymptotics for $k_j(\lambda|y|)$ and small variable asymptotics for $k'_i(\lambda|x|)$. The first integral is then bounded by
\begin{gather}
    \int_0^{|y|^{-1}} \lambda^{\gamma_1+1} (\lambda |x|)^{1-d_i} (\lambda|y|)^{(2-d_j)\wedge 0} d\lambda = |x|^{1-d_i}|y|^{-\gamma_1-3+d_i} \quad  j=1,2.
\end{gather}
The second integral can be estimated by
\begin{gather}
    \int_{|y|^{-1}}^{|x|^{-1}}(\lambda|x|)^{1-d_i}(\lambda|y|)^{\frac{1-d_j}{2}}e^{-\lambda|y|} d\lambda \lesssim |x|^{1-d_i}|y|^{-\gamma_1-3+d_i} \quad  j=1,2.
\end{gather}
In summary, the low energy part has upper bound
\begin{equation}
    \begin{cases}
        \begin{cases}
            |x|^{-\gamma_1-2}, & j=1,\\
            |x|^{-\gamma_1-4+d_j}|y|^{2-d_j}, & j=2.
        \end{cases}, & \textit{if}\quad |x|\ge |y|,\\
        |x|^{1-d_i} |y|^{-\gamma_1-3+d_i}, & \textit{if}\quad |x|\le |y|.
    \end{cases}
\end{equation}

\begin{remark}
We note that the above estimate applies in all cases except when $1<d_1<d_2=2$, because $k_2$ exhibits logarithmic behavior for small arguments. Nonetheless, we assert that even in this scenario, the same method can still be applied. In particular, we control the logarithmic term by bounding it with $C_\epsilon (x/y)^\epsilon$ for any $\epsilon>0$.
\end{remark}

For readers convinence, we summarize estimates for each quadrant into tables in Appendix~\ref{appendix}.

Next, we turn to the high energy terms: $I_2,II_2,III_2,IV_2$. In this setting, it suffices to focus on the integral:
\begin{equation}\label{eq_high}
    \int_{1/(|x|\wedge |y|)}^\infty \lambda F(\lambda) k'_{i}(\lambda|x|) k_j(\lambda|y|) d\lambda \quad i,j\in \{1,2\}.
\end{equation}
We observe that for all $\lambda>1/(|x|\wedge |y|)$, $|F(\lambda)|\lesssim \lambda^{\gamma_2} e^{2\lambda}$. So, by using large asymptotics for both $k_i'(\lambda|x|)$ and $k_j(\lambda|y|)$, \eqref{eq_high} is bounded by
\begin{gather}\label{eq_hh}
    |x|^{\frac{1-d_i}{2}}|y|^{\frac{1-d_j}{2}} \int_{1/(|x|\wedge |y|)}^\infty \lambda^{\gamma_2+2-\frac{d_i+d_j}{2}} e^{-\lambda(|x|+|y|-2)}d\lambda.
\end{gather}
Now, if $\gamma_2+2-\frac{d_i+d_j}{2}\ge 0$, we obsorbe the term $\lambda^{\gamma_2+2-\frac{d_i+d_j}{2}}$ into the exponential term and bound \eqref{eq_hh} by
\begin{equation}
    |x|^{\frac{1-d_i}{2}}|y|^{\frac{1-d_j}{2}} \int_{1/(|x|\wedge |y|)}^\infty e^{-c\lambda(|x|+|y|-2)}d\lambda \lesssim |x|^{\frac{1-d_i}{2}}|y|^{\frac{1-d_j}{2}} \frac{e^{-c\frac{|x|+|y|-2}{|x|\wedge |y|}}}{|x|+|y|-2}
\end{equation}
for some constant $c>0$. While if $\gamma_2+2-\frac{d_i+d_j}{2}< 0$, we bound $\lambda^{\gamma_2+2-\frac{d_i+d_j}{2}}$ by $(|x|\wedge |y|)^{-\gamma_2-2+\frac{d_i+d_j}{2}}$ and \eqref{eq_hh} can be then estimated by
\begin{equation}\label{eq_hhh}
    |x|^{-a}|y|^{-b} \frac{e^{-c\frac{|x|+|y|-2}{|x|\wedge |y|}}}{|x|+|y|-2}
\end{equation}
for some $a,b>0$. We use \eqref{eq_hhh} as a general upper bound for $T_H(x,y)$.

\subsection{Hardy-Hilbert Type Inequalities}

In this subsection, we examine the $L^p \textit{-}$boundedness of "\textit{Hardy-Hilbert type}" operators that act between measure spaces $([1,\infty),d\mu_1:=r^{n_1-1}dr)$ and $([1,\infty),d\mu_2:=r^{n_2-1}dr)$, where $n_1,n_2>1$. We also present some Lorentz-type endpoint estimates. These so-called “Hardy–Hilbert type” operators play a key role in understanding the boundedness of $\nabla \Delta^{-1/2}$.

Recall that a kernel $K(x,y)$ is said to be homogeneous of degree $\delta \in \mathbb{R}$ if 
\begin{equation*}
    K(\lambda x,\lambda y) = \lambda^{\delta} K(x,y) \quad \forall \lambda>0, \quad \forall x,y.
\end{equation*}
Integral operators with homogeneous kernel are systematically studied in \cite{HLP}. One can check from the tables in the Appendix~\ref{appendix} that all kernels in our context are homogeneous. We begin by generalizing \cite[Theorem~319]{HLP}.

\begin{lemma}\label{thmHLP}
Let $n_1,n_2,p>1$, and $K$ is an integral operator with kernel $K(x,y)$. Assume that $K(x,y)$ is non-negative with homogeneous degree of $-\delta$, where $\delta=n_2/p + n_1/{p'}$. Suppose that
\begin{equation*}
    \int_0^\infty K(x,1)x^{n_2/p -1}dx = \int_0^\infty K(1,y)y^{n_1/{p'}-1}dy <\infty.
\end{equation*}
Then, $K$ acts as a bounded operator from $L^p(\mathbb{R}^+,r^{n_1-1}dr)\to L^p(\mathbb{R}^+,r^{n_2-1}dr)$.
\end{lemma}

\begin{proof}
Let $g(x) = x^{n_1/p}f(x)$ and $h(x) = x^{\delta-n_1/{p'}}K(x,1)$. Then
\begin{align*}
    K(f)(x) &= \int_0^\infty K(x,y) f(y) y^{n_1} \frac{dy}{y} = \int_0^\infty y^{n_1/{p'}-\delta}K(x/y,1) \left(y^{n_1/p} f(y)\right) \frac{dy}{y}\\
    &= x^{n_1/{p'}-\delta} \int_0^\infty h(x/y) g(y) \frac{dy}{y} = x^{n_1/{p'}-\delta} (h\Tilde{*}g)(x),
\end{align*}
where $\Tilde{*}$ denotes the convolution on multiplicative group $(\mathbb{R}^+, \frac{dr}{r})$.

Therefore, by Young's convolution inequality (see for example \cite[Section 1.2]{G})
\begin{align*}
    \|K(f)\|_{L^p(\mathbb{R}^+,r^{n_2-1}dr)} &= \left(\int_0^\infty x^{p(n_1/{p'}-\delta)+n_2}|(h\Tilde{*}g)(x)|^p \frac{dx}{x} \right)^{1/p}\\
    &= \|h\Tilde{*}g\|_{L^p(\mathbb{R}^+,\frac{dx}{x})} \le \|h\|_{L^1(\mathbb{R}^+,\frac{dx}{x})} \|g\|_{L^p(\mathbb{R}^+,\frac{dx}{x})}\\
    &= \|h\|_{L^1(\mathbb{R}^+,\frac{dx}{x})} \|f\|_{L^p(\mathbb{R}^+,r^{n_1-1}dr)}.
\end{align*}
The proof is complete since
\begin{equation*}
    \|h\|_{L^1(\mathbb{R}^+,\frac{dx}{x})} = \int_0^\infty x^{\delta-n_1/{p'}-1}K(x,1)dx <\infty.
\end{equation*}

\end{proof}

Next, we consider an integral operator $K$ acting between spaces $([1,\infty),d\mu_1)$ and $ ([1,\infty),d\mu_2)$, where $d\mu_j(r) = r^{n_j-1}dr$ for $j=1,2$. The kernel is given by 
\begin{equation}\label{eq_K}
    K(x,y) = 
    \begin{cases}
       x^{-\alpha}y^{-\beta}, & x\le y,\\
       x^{-\alpha'}y^{-\beta'}, & x> y
    \end{cases}
\end{equation}
with $\alpha,\alpha',\beta,\beta' \ge 0$. We decompose this operator into two parts, corresponding to the two regions $x\le y$ and $x>y$
\begin{gather*}
    R_1(f)(x) = x^{-\alpha} \int_x^\infty y^{-\beta}f(y)d\mu_1(y) \quad
    R_2(f)(x) = x^{-\alpha'}\int_1^x y^{-\beta'}f(y)d\mu_1(y). 
\end{gather*}
Note that $R_1$ has adjoint operator given by
\begin{equation*}
    R_1^*(f)(x) = x^{-\beta}\int_1^x y^{-\alpha}f(y)d\mu_2(y),
\end{equation*}
which coincides with $R_2$ upon the substitutions $\alpha \to \beta'$, $\beta \to \alpha'$, $n_1\to n_2$ and $n_2\to n_1$. Operators of the form \eqref{eq_K} are also studied in \cite[Proposition 5.1]{GH1}, \cite[Lemma 5.4]{HS1D} and \cite[Lemma 2.6]{N}.

\begin{lemma}\cite[Lemma 2.6]{N,nix2019resolvent}\label{le_N}
Let $K$ be an integral operator defined by \eqref{eq_K}. If $p(\alpha+\beta-n_1)>n_2-n_1$, $p(\alpha'+\beta'-n_1)>n_2-n_1$ and
\begin{equation*}
    \frac{n_2}{n_2 \wedge \alpha'} < p < \frac{n_1}{0 \vee (n_1-\beta)},
\end{equation*}
then $K$ is bounded from $L^p([1,\infty),d\mu_1)\to L^p([1,\infty),d\mu_2)$.

\end{lemma}

To deal with the endpoint, we prove the following lemmata in the setting of Lorentz spaces.

\begin{lemma}\label{thmRW}
 For $0<\beta<n_1$ and $\alpha>0$, $R_1$ is bounded from (of restricted weak type $(p,q)$)
 \begin{equation*}
     L^{p,1}([1,\infty),d\mu_1)\to L^{q,\infty}([1,\infty),d\mu_2)
 \end{equation*}
for $1<p\le \frac{n_1}{n_1-\beta}$ and $q\ge \frac{n_2}{\alpha}$.

\end{lemma}

\begin{proof}
 Let $x\ge 1$, and $f\in L^{p,1}$. Set $\mathcal{X}_{[x,\infty)}$ to be the characteristic function on interval $[x,\infty)$. By Hardy-Littlewood inequality,
\begin{gather*}
    |R_1(f)(x)|\le x^{-\alpha} \int_1^\infty |f(y)| y^{-\beta}\mathcal{X}_{[x,\infty)}(y) d\mu_1(y)\\
    \le x^{-\alpha} \|f\|_{(p,1)} \|F_x\|_{(p',\infty)}
\end{gather*}
where $F_x(y) = y^{-\beta}\mathcal{X}_{[x,\infty)}(y)$. Note that a direct computation gives
\begin{equation*}
    d_{F_x}(\lambda) = \mu_1\left(\{y\ge x: y^{-\beta}>\lambda\}\right)\simeq 
    \begin{cases}
        \lambda^{-\frac{n_1}{\beta}}-x^{n_1}, & \lambda <x^{-\beta}\\
        0, & \lambda \ge x^{-\beta}.
    \end{cases}
\end{equation*}
Therefore,
\begin{align*}
    \|F_x\|_{(p',\infty)} &= \sup_{\lambda>0}  \lambda d_{F_x}(\lambda)^{1/{p'}} \\
    &\simeq \left(\sup_{0<\lambda<x^{-\beta}} \lambda^{p'-d_1/\beta}- x^{d_1} \lambda^{p'}\right)^{1/{p'}}\\
    &\lesssim \sup_{0<\lambda <x^{-\beta}} \lambda^{1-\frac{n_1}{\beta p'}}.
\end{align*}
Now, since $p' \ge \frac{n_1}{\beta}$, we have $1-\frac{n_1}{\beta p'}\ge 0$ and 
\begin{equation*}
    \|F_x\|_{(p',\infty)}\lesssim x^{\beta \left(\frac{n_1}{\beta p'}-1\right)}\lesssim 1
\end{equation*}
uniformly in $x\ge 1$. Thus we have pointwise estimate
\begin{equation*}
    |R_1(f)(x)|\lesssim \|f\|_{(p,1)} x^{-\alpha}.
\end{equation*}
Now, since $x^{-\alpha} \in L^{q,\infty}(d\mu_2)$ if $q\ge n_2/\alpha$, the result follows. 
\end{proof}

\begin{remark}\label{endpoint}
Note that $x^{-\alpha} \in L^{q,1}(d\mu_2)$ if $q> n_2/\alpha$. So if $\frac{n_2}{\alpha}<p\le \frac{n_1}{n_1-\beta}$, then $R_1$ is bounded on $L^{p,1}$.
\end{remark}

\begin{lemma}\label{le_R2}
If $\beta'\ge 0$ and $\alpha'>0$, then $R_2$ is bounded from (of restricted weak type $(p,q)$)
\begin{equation*}
    L^{p,1}([1,\infty), d\mu_1) \to L^{q,\infty}([1,\infty), d\mu_2)
\end{equation*}
provided
\begin{equation*}
q\ge 
    \begin{cases}
        \frac{n_2}{\alpha'}, & \textit{If} \quad p'\ge \frac{n_1}{\beta'},\\
        \frac{n_2}{\alpha'+\beta'-n_1/p'}, & \textit{If} \quad p'< \frac{n_1}{\beta'}.
    \end{cases}
\end{equation*}
\end{lemma}

\begin{proof}
Similar to the proof of Lemma \ref{thmRW}, we have by Hardy-Littlewood inequality
\begin{equation*}
    |R_2(f)(x)| \le x^{-\alpha'} \|f\|_{(p,1)} \|g\|_{(p',\infty)},
\end{equation*}
where $g(y) = y^{-\beta'}\mathcal{X}_{[1,x]}(y)$. Now, a straightforward computation yields
\begin{equation*}
    d_g(\lambda) \lesssim \begin{cases}
        x^{n_1}, & 0<\lambda<x^{-\beta'},\\
        \lambda^{-n_1/\beta'}, & x^{-\beta'}\le \lambda<1,\\
        0, & \lambda\ge 1,
    \end{cases}
\end{equation*}
which implies
\begin{gather*}
    \|g\|_{(p',\infty)} \lesssim \begin{cases}
        1, & p'\ge \frac{n_1}{\beta'},\\
        x^{n_1/{p'}-\beta'}, & p'< \frac{n_1}{\beta'}.
    \end{cases}
\end{gather*}
The result follows immediately.
\end{proof}

Finally, we state the following lemma to manage the kernel of $T_H$ in \eqref{eq_hhh} (essentially a reformulation of \cite[Lemma 2.7]{N,nix2019resolvent} and\cite[Theorem 5.5]{HS1D}).

\begin{lemma}\label{le_TH}
Let $n_1,n_2>1$ and $K$ be an integral operator with kernel:
\begin{equation*}
    K(x,y) = x^{-a}y^{-b} \frac{e^{-c\frac{x+y-2}{x\wedge y}}}{x+y-2}, \quad x,y\ge 1, \quad a,b>0.
\end{equation*}
Then, $K$ is bounded from $L^p([1,\infty),d\mu_1)\to L^p([1,\infty),d\mu_2)$ for all $p\in (1,\infty)$.
\end{lemma}

\begin{proof}
For the case $2\le x+y\le 4$, this kernel is essentially bounded by the kernel $\frac{1}{s+t}$ acting on $\mathbb{R}^+$. It then follows by \cite[Section 9.1]{HLP} that this part of kernel acts boundedly from $L^p([1,\infty),d\mu_1)\to L^p([1,\infty),d\mu_2)$ for $p\in (1,\infty)$. Regarding the case where $x+y\ge 4$, the exponential term in the kernel can be bounded by $\min ((x/y)^{-N}, (y/x)^{-N})$ for any $N>0$ which again acts boundedly from $L^p([1,\infty),d\mu_1)\to L^p([1,\infty),d\mu_2)$ for $p\in (1,\infty)$ by say, Hölder's inequality.
\end{proof}

\subsection{Proof of Theorem~\ref{thm_main}}

This section is devoted to prove the main result, Theorem \ref{thm_main}. We begin with the positive result. Using notations from Section~\ref{sec_2}, by \eqref{riesz decomposition}, the Riesz transform can be decomposed into pieces:
\begin{equation*}
    \nabla \Delta^{-1/2} = T_L + T_H + R_{kl}.
\end{equation*}
We start by studying the "$kl$" part i.e.
\begin{equation}
    R_{kl} = \int_0^\infty \nabla (\Delta+k^2)^{-1}_{kl}dk.
\end{equation}

Let $1<d_1<d_2$. Observe that by \eqref{eq_y>1} and \eqref{eq_y<1}, the "$kl$" part only supports in $Q_1$ or $Q_3$. That is either $x,y\ge 1$ or $x,y\le -1$ (i.e. the "dimension" does not change after mapping). As a consequence, $R_{kl}$, by \cite[Theorem 5.1, Remark 5.2]{HS1D}, is bounded on $L^p(\Tilde{\mathbb{R}},d\mu)$ for $p\in (1,\infty)$. In addition, $R_{kl}$ is bounded from $L^{p_0,1}\to L^{p_0,\infty}$ by interpolation.

Next, by Lemma \ref{le_TH} and estimate \eqref{eq_hhh}, we conclude that for all $1<d_1<d_2$, $T_H$ is bounded on $L^p(\Tilde{\mathbb{R}},d\mu)$ for $p\in (1,\infty)$. In addition, $T_H$ is of restricted weak type $(p_0,p_0)$ by interpolation.

It remains to verify the boundedness of $T_L$. By notations from last section, $T_L$ can be further decomposed into
\begin{equation*}
    T_L = I_1 + II_1 + III_1 + IV_1,
\end{equation*}
and it is enough to examine them individually. The strategy is to apply Lemma \ref{le_N} to $I_1,II_1,IV_1$ and to employ Lemma \ref{thmHLP} to $III_1$. Since the analysis for all cases are similar, we only give details for the case $1<d_1<2<d_2$ as an example.

$\bullet$ $\textbf{$I_1$.}$ From tables in Appendix~\ref{appendix}, we know that for $x,y\ge 1$
\begin{equation*}
    |I_1(x,y)|\lesssim \begin{cases}
        x^{1-d_2} y^{d_1-d_2-1}, & 1\le x\le y,\\
        x^{d_1-d_2-2} y^{2-d_2}, & x>y.
    \end{cases}
\end{equation*}
Then, by Lemma \ref{le_N} with $n_1=n_2=d_2$, $\alpha=d_2-1$, $\beta=d_2-d_1+1$, $\alpha'=d_2-d_1+2$ and $\beta'=d_2-2$, we conclude that $I_1$ is bounded on $L^p([1,\infty),r^{d_2-1}dr)$ for
\begin{equation}\label{I}
    1 =\frac{d_2}{d_2 \wedge (d_2-d_1+2)} < p < \frac{d_2}{0 \vee (d_1-1)} = \frac{d_2}{d_1-1},
\end{equation}
since $p(\alpha+\beta-n_1)=p(\alpha'+\beta'-n_1)=p(d_2-d_1)>(d_2-d_2)=0$. In addition, by Lemma~\ref{thmRW} and Lemma \ref{le_R2}, $I_1$ is of restricted weak type $(d_1',d_1')$.

$\bullet$ $\textbf{$II_1$.}$ For $x\le -1$ and $y\ge 1$, 
\begin{equation*}
    |II_1(x,y)|\simeq \begin{cases}
        |x|^{1-d_1} y^{d_1-d_2-1}, & 1\le |x|\le y,\\
        |x|^{-2} y^{2-d_2}, & |x|>y.
    \end{cases}
\end{equation*}
Now, by Lemma \ref{le_N} with $n_1=d_2$, $n_2=d_1$, $\alpha=d_1-1$, $\beta=d_2-d_1+1$, $\alpha'=2$ and $\beta'=d_2-2$, we assert that $II_1$ is bounded from $L^p([1,\infty),r^{d_2-1}dr)\to L^p((-\infty,-1],r^{d_1-1}dr)$ for
\begin{equation}\label{II}
    1 =\frac{d_1}{d_1 \wedge 2} < p < \frac{d_2}{0 \vee (d_1-1)} = \frac{d_2}{d_1-1},
\end{equation}
since $p(\alpha+\beta-n_1)=p(\alpha'+\beta'-n_1)=0>d_1-d_2$. In addition, by Lemma \ref{thmRW} and Lemma~\ref{le_R2}, $II_1$ is of restricted weak type $(d_1',d_1')$.

$\bullet$ $\textbf{$III_1$.}$ For $x,y\le -1$, 
\begin{equation*}
    |III_1(x,y)|\lesssim \begin{cases}
        |x|^{1-d_1} |y|^{-1}, & 1\le |x|\le |y|,\\
        |x|^{-d_1}, & |x|>|y|.
    \end{cases}
\end{equation*}
Next, by Lemma \ref{thmHLP} with $n_1=n_2=d_1$ and $\delta=d_1$, we deduce that $III_1$ is bounded on $L^p((-\infty,-1],r^{d_1-1}dr)$ for
\begin{equation}\label{III}
    1 < p < d_1',
\end{equation}
since for $p$ in the above range
\begin{equation*}
    \int_0^\infty |III_1(x,1)| x^{d_1/p-1}dx \lesssim \int_0^1 x^{-d_1/p'}dx + \int_1^\infty x^{-d_1/p'-1}dx <\infty.
\end{equation*}
In addition, by Lemma \ref{thmRW} and Lemma \ref{le_R2}, $III_1$ is of restricted weak type $(d_1',d_1')$.

$\bullet$ $\textbf{$IV_1$.}$ For $x\ge 1$ and $y\le -1$, 
\begin{equation*}
    |IV_1(x,y)|\simeq \begin{cases}
        x^{1-d_2} |y|^{-1}, & 1\le x\le |y|,\\
        x^{-d_2}, & x>|y|.
    \end{cases}
\end{equation*}
Now, by Lemma \ref{le_N} with $n_1=d_1$, $n_2=d_2$, $\alpha=d_2-1$, $\beta=1$, $\alpha'=d_2$ and $\beta'=0$, we claim that $IV_1$ is bounded from $L^p((-\infty,-1],r^{d_1-1}dr)\to L^p([1,\infty),r^{d_2-1}dr)$ for
\begin{equation}\label{IV}
    1 =\frac{d_2}{d_2 \wedge d_2} < p < \frac{d_1}{0 \vee (d_1-1)} = d_1',
\end{equation}
since $p(\alpha+\beta-n_1)=p(\alpha'+\beta'-n_1)=p(d_2-d_1)>d_2-d_1$. In addition, by Lemma \ref{thmRW} and Lemma \ref{le_R2}, $IV_1$ is of restricted weak type $(d_1',d_1')$.

\medskip
For convinence, we summarise the range of $L^p\textit{-}$boundedness for all cases in the following table.
\begin{table}[!htbp]
    \centering
    \begin{tabular}{|c|c|c|c|c|c|}
    \hline
    \diagbox{Q}{C}     & $1<d_1<d_2<2$ & $1<d_1<d_2=2$ & $1<d_1<2<d_2$ & $2=d_1<d_2$ & $2<d_1<d_2$  \\
    \hline
    $I_1$ & $(1,\frac{d_2}{d_1-1})$     & $(1,\frac{2}{d_1-1})$ & $(1,\frac{d_2}{d_1-1})$
     & $(1,d_2)$ & $(1,d_2)$
    \\
    \hline
    $II_1$ &  $(1,\frac{d_2}{d_1-1})$ & $(1,\frac{2}{d_1-1})$  &  $(1,\frac{d_2}{d_1-1})$
     & $(1,d_2)$ & $(1,d_2)$
     \\
    \hline
    $III_1$ & $(1,d_1')$ &$(1,d_1')$   &  $(1,d_1')$
     &$(1,2)$ & $(1,d_1)$
     \\
    \hline
    $IV_1$ &  $(1,d_1')$ &$(1,d_1')$  &  $(1,d_1')$
     &$(1,2)$ &$(1,d_1)$\\
    \hline
    $T_L$& $(1,d_1')$ &$(1,d_1')$ &$(1,d_1')$ &$(1,2)$ &$(1,d_1)$\\
    \hline
    \end{tabular}
    \caption{}
    \label{tab:my_label}
\end{table}

Therefore, combining the results of $T_L$, $T_H$ and $R_{kl}$, we conclude that for any pair $d_1,d_2>1$, the Riesz transform on broken line, $\nabla \Delta^{-1/2}$, is bounded on $L^p(\Tilde{\mathbb{R}},d\mu)$ for
\begin{equation*}
    1< p < d_* \vee d_*' := p_0\quad \textrm{or}\quad p=2,
\end{equation*}
where $d_*= d_1 \wedge d_2$. In addition, $\nabla \Delta^{-1/2}$ is of restricted weak type $(p_0,p_0)$.
\medskip

To complete our argument, we also need to show that the Riesz transform is $L^p\textit{-}$unbounded for $p\ge p_0$. We capture this requirement in the following lemma.

\begin{lemma}\label{lemma_Negative}
Under the assumptions of Theorem~\ref{thm_main}, $\nabla \Delta^{-1/2}$ is unbounded on $L^p$ for all $p\ge p_0$, where $p_0 = d_* \vee d_*'$.
\end{lemma}

\begin{proof}
Since $\nabla \Delta^{-1/2} = T_L + T_H + R_{kl}$ and 
\begin{equation*}
    \|T_H+R_{kl}\|_{p\to p} \lesssim 1 \quad \forall 1<p<\infty.
\end{equation*}
It is enough to verify the discontinuity of $T_L$ on $L^p$ for $p\ge p_0$.

Next, we know that $T_L = I_1+II_1+III_1+IV_1$. It is plain that the negative result follows if we can show one of them is unbounded on $L^{p_0}$ since they have disjoint supports. Note that for $II_1$ and $IV_1$, the upper bound from the tables in Appendix~\ref{appendix} is also a lower bound since the coefficient $A(\lambda)$ is positive for all $\lambda \le 1$. Observe that $IV_1$ is bounded on $L^p$ for $p\in (1,p_0)$ for all cases. Therefore, it suffices to treat $IV_1$ and to prove the unboundedness of $\mathcal{X}_{[1,\infty)}\circ T_L \circ \mathcal{X}_{(-\infty,-1]}$. Note that for all cases we consider, when $1\le x\le |y|$, 
\begin{equation*}
    |IV_1(x,y)|\simeq  x^{-\alpha} |y|^{-\beta}, \quad \textrm{for some}\quad \alpha>0,
\end{equation*}
where $\beta =1$ if $d_1<2$ and $\beta=d_1-1$ if $d_1>2$. Hence, we only need to confirm the discontinuity of the following operator on $L^{p_0}$: 
\begin{equation}\label{C}
    f \mapsto (\cdot)^{-\alpha} \int_{(\cdot)}^\infty y^{-\beta} f(y) y^{d_1-1} dy.
\end{equation}
Let $f(y) = y^{\beta-d_1}(1+\log y)^{-1}$. Since $\beta<d_1$, it is plain that
\begin{gather*}
    \|f\|_{p_0}^{p_0} = \int_1^\infty y^{-p_0(d_1-\beta)}(1+ \log y)^{-p_0}y^{d_1-1}dy
    = \int_0^\infty e^{-cx}(1+x)^{-p_0}dx<\infty,
\end{gather*}
since 
\begin{equation*}
    c = p_0(d_1-\beta)-d_1 = \begin{cases}
        (d_1' \vee d_1)(d_1-1)-d_1 = 0, & 1<d_1<2,\\
        (d_1' \vee d_1)-d_1 = 0, & d_1>2.
    \end{cases}
\end{equation*}
However, for all $ x\ge 1$, \eqref{C} equals
\begin{equation*}
    x^{-\alpha}\int_x^\infty y^{-1}(1+\log y)^{-1}dy = x^{-\alpha}\int_{\log x}^\infty (1+t)^{-1}dt = \infty.
\end{equation*}
The result follows easily.
\end{proof}

The proof of Theorem \ref{thm_main} is now complete.

\section{Reverse Riesz inequality}

In this section, we aim to study the reverse Riesz problem. That is we want to establish inequality:
\begin{align*}
    \|\Delta^{1/2}f\|_p \lesssim \|\nabla f\|_p 
\end{align*}
for some range of $p\in (1,\infty)$.

By \cite{H2}, the author develops a so-called harmonic annihilation method to prove such type of inequalities on a class of non-doubling manifolds.

\begin{subsection}{Hardy's inequality}

Following the method in \cite{H2}, the first step is to establish a Hardy-type inequality. 

\begin{lemma}\label{le_Hardy}
Let $1<d_1<d_2$. The following Hardy-type inequality holds
\begin{align*}
    \int_{\Tilde{\mathbb{R}}} \frac{|f(x)|^p}{|x|^p} d\mu \lesssim \int_{\Tilde{\mathbb{R}}} |f'(x)|^p d\mu, \quad \forall f\in S_0(\Tilde{\mathbb{R}})
\end{align*}
for all $1\le p< d_*$, where $d_* = d_1 \wedge d_2$.
\end{lemma}

\begin{proof}
For $0<b<\infty$ and $1\le p<\infty$, it is well-known that (see for example \cite[Exercise 1.2.8]{G})
\begin{align*}
    \int_0^\infty \left| \int_x^\infty u(t) dt \right|^p x^{b-1} dx \le \left(\frac{p}{b}\right)^p \int_0^\infty |u(t)|^p t^{p+b-1} dt.
\end{align*}
Now, let $f$ be a function supported in $[1,\infty)$ such that $f(x) = \int_x^\infty u(t) dt$ and $p+b=d$. The above inequality implies
\begin{align*}
 \int_1^\infty \left|\frac{f(x)}{x}\right|^p x^{d-1} dx \le \left(\frac{p}{b}\right)^p \int_1^\infty |f'(t)|^p t^{d-1} dt,
\end{align*}
provided $1\le p<d$. Immediately, one infers Hardy's inequality on broken line:
\begin{align}\label{Hardy_1D}
    \left\| \frac{f(\cdot)}{|\cdot|} \right\|_{L^p(\Tilde{\mathbb{R}}, d\mu)} \lesssim \|f'\|_{L^p(\Tilde{\mathbb{R}}, d\mu)}, \quad \forall 1\le p< d_*,
\end{align}
for all $f\in S_0(\Tilde{\mathbb{R}})$.
\end{proof}

\end{subsection}

\begin{subsection}{Proof of Theorem~\ref{thm_RR_main}}
In this subsection, we prove Theorem~\ref{thm_RR_main}. Recall notations from Section~\ref{sec_2}. By Theorem~\ref{thm_main} and duality, we only need to consider the range
\begin{align*}
    p\in \begin{cases}
        (1,d_*), & \textit{if}\quad 1<d_*<2,\\
        (1,2), & \textit{if}\quad d_*=2,\\
        (1,d_*'], & \textit{if}\quad d_*> 2.
    \end{cases}
\end{align*}
We use the following resolution to identity:
\begin{equation*}
    \Delta^{1/2} = \int_0^\infty \Delta (\Delta+\lambda^2)^{-1}d\lambda.
\end{equation*}
Let $f\in S_0(\Tilde{\mathbb{R}})$ and $g\in C_c^\infty(\Tilde{\mathbb{R}})$. By the self-adjointness and positivity of $\Delta$, 
\begin{align*}
    \langle \Delta^{1/2}f ,g \rangle &= \left \langle \int_0^\infty \Delta (\Delta+\lambda^2)^{-1}f d\lambda, g \right \rangle = \left\langle \nabla f, \nabla \int_0^\infty (\Delta+\lambda^2)^{-1}g d\lambda \right\rangle\\
    &= \left\langle \nabla f, \nabla \int_0^\infty (\Delta+\lambda^2)_{kk}^{-1}g d\lambda \right\rangle + \left\langle \nabla f, \nabla \int_0^\infty (\Delta+\lambda^2)_{kl}^{-1}g d\lambda \right\rangle.
\end{align*}
By \cite[Theorem 5.1, Remark 5.2]{HS1D}, the "$kl$" part of Riesz transform is bounded on $L^q$ for all $1<q<\infty$. It is then clear that
\begin{equation*}
    \left|\left\langle \nabla f, \nabla \int_0^\infty (\Delta+\lambda^2)_{kl}^{-1}g d\lambda \right\rangle \right| \lesssim  \|f'\|_p \|R_{kl}g\|_{p'} \lesssim \|f'\|_p \|g\|_{p'}. 
\end{equation*}
Hence, it is sufficient to consider the "$kk$" part in the inner product. Let $F(\cdot)$ be, $A(\cdot)$, $B(\cdot)$ or $C(\cdot)$. Set $\mathcal{X}_1$, $\mathcal{X}_2$ to be the characterization functions on intervals $(-\infty,-1]$ and $[1,\infty)$ respectively. We write 
\begin{align*}
    \nabla \int_0^\infty (\Delta+\lambda^2)_{kk}^{-1}g(x)d\lambda &= \sum_{i,j=1}^2 \int_{\Tilde{\mathbb{R}}} \int_0^\infty \lambda k_i'(\lambda|x|) \mathcal{X}_i(x) k_j(\lambda|y|) \mathcal{X}_j(y) F(\lambda)d\lambda g(y)d\mu(y) \\
    &= \sum_{i,j=1}^2 \int_{\Tilde{\mathbb{R}}} \int_{1}^\infty  + \sum_{i,j=1}^2 \int_{\Tilde{\mathbb{R}}} \int_0^{1} := T_hg + \sum_{i,j=1}^2 \int_{\Tilde{\mathbb{R}}} \int_0^{1},
\end{align*}
where $i,j\in \{1,2\}$.

Now, by estimates from Section~\ref{sec_2} and Lemma~\ref{le_TH}, the high energy part $T_h$ is bounded on all $L^q$. Therefore,
\begin{equation*}
    |\langle \nabla f, T_hg \rangle |\lesssim \|f'\|_p \|g\|_{p'}.
\end{equation*}
Hence, it suffices to treat the following bilinear form:
\begin{align*}
    \mathcal{B}(f,g)&:= \sum_{i,j=1}^2 \left\langle \nabla f, \int_0^1 \nabla [k_i(\lambda|x|)] \mathcal{X}_i(x) \int_{\Tilde{\mathbb{R}}} k_j(\lambda|y|) \mathcal{X}_j(y) g(y)d\mu(y) F(\lambda) d\lambda \right\rangle\\
    &= \sum_{i,j=1}^2 \int_{\Tilde{\mathbb{R}}} f'(x) \int_0^1 \left[\frac{d}{dx}k_i(\lambda |x|)\right] \mathcal{X}_i(x) \int_{\Tilde{\mathbb{R}}} k_j(\lambda|y|) \mathcal{X}_j(y) g(y)d\mu(y) F(\lambda) d\lambda d\mu(x).
\end{align*}
Integrating by parts and using the fact \eqref{eq_ODE}:
\begin{align*}
    k_i^{''}(\lambda \cdot) + \frac{d_i-1}{r}k_i'(\lambda \cdot) = \lambda^2 k_i(\lambda \cdot),\quad i=1,2,
\end{align*}
we deduce that
\begin{align*}
    \mathcal{B}(f,g) &= - \sum_{i,j=1}^2 \int_{\Tilde{\mathbb{R}}} f(x) \int_0^1 \lambda^2 k_i(\lambda |x|) \int_{\Tilde{\mathbb{R}}} k_j(\lambda|y|)g(y)d\mu(y) F(\lambda) d\lambda d\mu(x).
\end{align*}
Define operator
\begin{align}
    \mathcal{T}_{ij}: h \mapsto |\cdot|\int_0^1 \lambda^2 F(\lambda) k_i(\lambda |\cdot|) \mathcal{X}_i(\cdot) \int_{\Tilde{\mathbb{R}}} k_j(\lambda|y|) \mathcal{X}_j(y) h(y)d\mu(y) d\lambda, \quad h\in C_c^\infty(\Tilde{\mathbb{R}}).
\end{align}
By  Hölder's inequality, for $1<q<\infty$,
\begin{align*}
    |\mathcal{T}_{ij}g(x)|\le \|g\|_{q'} |x| \mathcal{X}_i(x) \int_0^1 \lambda^{2} |F(\lambda)| |k_i(\lambda |x|)| \left[\int_{\Tilde{\mathbb{R}}} k_j(\lambda|y|)^q \mathcal{X}_j(y) d\mu(y)\right]^{\frac{1}{q}} d\lambda.
\end{align*}
By estimates from Section~\ref{sec_2}, one can assume
\begin{align*}
    |F(\lambda)| \lesssim \lambda^{\sigma}, \quad \forall \lambda \le 1,
\end{align*}
and
\begin{align*}
    |k_i(r)| \lesssim \begin{cases}
        \begin{cases}
            1, & 1<d_i<2,\\
            1-\log{r}, & d_i=2,\\
            r^{2-d_i}, & d_i>2,
        \end{cases} &r\le 1,\\
        r^{\frac{1-d_i}{2}} e^{-r}, &r\ge 1,
    \end{cases}
\end{align*}
where $\sigma \in \mathbb{R}$ depending on which quadrant we are studying.

\begin{lemma}\label{le_key}
Let $1<d_1<d_2$. Then, for all $1\le i,j\le 2$,
\begin{align*}
    \|\mathcal{T}_{ij}g\|_{q'} \lesssim \|g\|_{q'},\quad 1<q<\infty,
\end{align*}
for all $g\in C_c^\infty(\Tilde{\mathbb{R}})$.
\end{lemma}

\begin{proof}
We start by estimating the integral:
\begin{align*}
\mathcal{I}_j(\lambda):=\left[\int_{\Tilde{\mathbb{R}}} k_j(\lambda|y|)^q \mathcal{X}_j(y) d\mu(y)\right]^{\frac{1}{q}},\quad \lambda \le 1.
\end{align*}
Directly, we confirm
\begin{align*}
    \mathcal{I}_j(\lambda)^q &\lesssim  \begin{cases}
        \int_1^{\lambda^{-1}} y^{d_j-1} dy, & 1<d_j< 2,\\
        \int_1^{\lambda^{-1}}(1-\log{\lambda y})^q y^{d_j-1} dy, & d_j=2,\\
        \int_1^{\lambda^{-1}}(\lambda y)^{q(2-d_i)} y^{d_j-1} dy, & d_j> 2,
    \end{cases}\\
    &+ \int_{\lambda^{-1}}^\infty (\lambda y)^{\frac{1-d_j}{2}q} e^{-q\lambda y} y^{d_j-1} dy.
\end{align*}
One checks easily that the second term can be bounded by
\begin{align*}
    \int_{\lambda^{-1}}^\infty (\lambda y)^{\frac{1-d_j}{2}q} e^{-q\lambda y} y^{d_j-1} dy \lesssim \lambda^{-d_j}.
\end{align*}
As for the first term, a careful calculation gives an upper bound:
\begin{align*}
    \begin{cases}
        \lambda^{-d_j}, & 1<d_j\le 2,\\
        \begin{cases}
            \lambda^{-d_j}, & 1<q<\frac{d_j}{d_j-2},\\
            \lambda^{-d_j} \log{\lambda^{-1}}, & q=\frac{d_j}{d_j-2},\\
            \lambda^{q(2-d_j)}, & q>\frac{d_j}{d_j-2},
        \end{cases}\quad &d_j>2.
    \end{cases}
\end{align*}
Summarize the above estimates to get
\begin{align*}
    \mathcal{I}_j(\lambda) \lesssim \begin{cases}
        \lambda^{-d_j/q}, & 1<d_j\le 2,\\
        \begin{cases}
            \lambda^{-d_j/q}, & 1<q<\frac{d_j}{d_j-2},\\
            \lambda^{-d_j/q}(1+ \log{\lambda^{-1}})^{1/q}, & q=\frac{d_j}{d_j-2},\\
            \lambda^{2-d_j}, & q>\frac{d_j}{d_j-2},
        \end{cases}\quad &d_j>2.
    \end{cases}
\end{align*}
Next, we analyze each case separately.
\medskip

$\bullet$ \textit{Case 1. $1<d_1<d_2<2$}. In this case, we have $\sigma +2 = d_2$, $2d_2-d_1$ or $d_1$ and $|k_i(\lambda|x|)|\lesssim 1$. Hence,
\begin{align*}
    |\mathcal{T}_{ij}g(x)| &\lesssim \|g\|_{q'} |x| \mathcal{X}_{i}(x) \left[\int_0^{|x|^{-1}} \lambda^{2+\sigma-d_j/q} d\lambda + \int_{|x|^{-1}}^\infty \lambda^{2+\sigma-d_j/q} (\lambda |x|)^{\frac{1-d_i}{2}} e^{-\lambda|x|} d\lambda \right]\\
    &\lesssim \|g\|_{q'} |x| \mathcal{X}_{i}(x) \left[\int_0^{|x|^{-1}} \lambda^{2+\sigma-d_j/q} d\lambda + |x|^{-1+d_j/q - (2+\sigma)} \right].
\end{align*}
Note that $\mathcal{T}_{12}, \mathcal{T}_{21}$ correspond to coefficient $A(\lambda)$, which means for these two operators, we have $2+\sigma = d_2$. Therefore,
\begin{align*}
    |\mathcal{T}_{12}g(x)| \lesssim \|g\|_{q'} |x| \mathcal{X}_{1}(x) \left[\int_0^{|x|^{-1}} \lambda^{d_2-d_2/q} d\lambda + |x|^{-1+d_2/q - d_2} \right] \lesssim \|g\|_{q'} |x|^{-d_2/q'} \mathcal{X}_{1}(x),
\end{align*}
and
\begin{align*}
    |\mathcal{T}_{21}g(x)| \lesssim \|g\|_{q'} |x| \mathcal{X}_{2}(x) \left[\int_0^{|x|^{-1}} \lambda^{d_2-d_1/q} d\lambda + |x|^{-1+d_1/q - d_2} \right] \lesssim \|g\|_{q'} |x|^{d_1-d_2- d_1/q'} \mathcal{X}_{2}(x).
\end{align*}
Similarly, $\mathcal{T}_{11}, \mathcal{T}_{22}$ correspond to coefficients $C(\lambda)$ and $B(\lambda)$ respectively. Consequently,
\begin{align*}
    |\mathcal{T}_{11}g(x)|\lesssim \|g\|_{q'} |x|^{-d_1/q'} \mathcal{X}_{1}(x),\quad |\mathcal{T}_{22}g(x)|\lesssim \|g\|_{q'} |x|^{d_1-d_2-d_2/q'} \mathcal{X}_{2}(x).
\end{align*}
Now, for $\mathcal{T}_{12}, \mathcal{T}_{21}, \mathcal{T}_{22}$, we use Hölder's inequality to derive
\begin{align*}
    \|(\mathcal{T}_{12} + \mathcal{T}_{21} + \mathcal{T}_{22})g\|_{q'} \lesssim \|g\|_{q'},\quad \forall 1<q<\infty, 
\end{align*}
since $d_1<d_2$. While for $\mathcal{T}_{11}$, a weak type argument yields that for all $\delta>0$,
\begin{align*}
    \mu \left(\left\{ x\in \Tilde{\mathbb{R}}; |\mathcal{T}_{11}g(x)| > \delta   \right\}\right) &\lesssim \mu \left(\left\{ x\in (-\infty, -1]; C\|g\|_{q'} |x|^{-d_1/q'} > \delta   \right\}\right)\\
    &= \mu \left(\left\{ x\in (-\infty, -1]; |x| \le C \left(\frac{\|g\|_{q'}}{\delta}\right)^{q'/d_1}   \right\}\right)\\
    &\lesssim \left(\frac{\|g\|_{q'}}{\delta}\right)^{q'},
\end{align*}
i.e. $\mathcal{T}_{11}$ is of weak type $(q',q')$ for all $1<q<\infty$. By interpolation, we end up with conclusion:
\begin{align*}
    \|\mathcal{T}_{ij}g\|_{q'} \lesssim \|g\|_{q'},\quad 1<q<\infty,
\end{align*}
for all $1\le i,j\le 2$
\medskip

$\bullet$ \textit{Case 2. $1<d_1<d_2=2$}. We use the same method as in the \textit{Case 1}. We end up with estimates:
\begin{align*}
    |\mathcal{T}_{ij}g(x)| \lesssim \|g\|_{q'} \begin{cases}
        |x|^{-d_1/q'} \mathcal{X}_1(x), & i=1, j=1,\\
        |x|^{-2/q'} \mathcal{X}_1(x), & i=1, j=2,\\
        |x|^{d_1-2-d_1/q'}\mathcal{X}_2(x), & i=2, j=1,\\
        |x|^{d_1-2-2/q'} \mathcal{X}_2(x), & i=2, j=2.
    \end{cases}
\end{align*}
The $L^{q'}\textit{-}$boundedness of $\mathcal{T}_{12}, \mathcal{T}_{21}, \mathcal{T}_{22}$ are guaranteed by Hölder's inequality. While the boundedness of $\mathcal{T}_{11}$ is ensured by a weak type estimate and interpolation as showed in the \textit{Case~1}.

\medskip

$\bullet$ \textit{Case 3. $1<d_1<2<d_2$}. We have
\begin{align*}
    |\mathcal{T}_{ij}g(x)| \lesssim \|g\|_{q'} \begin{cases}
        |x|^{-d_1/q'} \mathcal{X}_1(x), & i=1, j=1,\\
        |x|^{d_1-d_2-d_1/q'}\mathcal{X}_2(x), & i=2, j=1.\\
    \end{cases}
\end{align*}
Moreover, for any $\epsilon>0$
\begin{align*}
    |\mathcal{T}_{ij}g(x)| \lesssim_{\epsilon} \|g\|_{q'} \begin{cases}
        \mathcal{X}_1(x) \begin{cases}
            |x|^{-d_2/q'}, & 1<q<\frac{d_2}{d_2-2},\\
            |x|^{-d_2/q'+\epsilon}, & q = \frac{d_2}{d_2-2},\\
            |x|^{-2}, & q> \frac{d_2}{d_2-2},
        \end{cases}, & i=1, j=2,\\
        \mathcal{X}_2(x) \begin{cases}
            |x|^{d_1-d_2-d_2/q'}, & 1<q<\frac{d_2}{d_2-2},\\
            |x|^{-2+\epsilon +d_1-d_2}, & q = \frac{d_2}{d_2-2},\\
            |x|^{d_1-d_2-2}, & q> \frac{d_2}{d_2-2},
        \end{cases}, & i=2, j=2,\\
    \end{cases}
\end{align*}
where the parameter $\epsilon$ comes from the simple estimate: 
\begin{align*}
    1-\log{\lambda} \le C_\epsilon \lambda^{-\epsilon},\quad \forall 0<\lambda\le 1.
\end{align*}
The rest of estimates are similar to cases as before, i.e. we use weak type argument to $\mathcal{T}_{11}$ and apply Hölder's inequality to $\mathcal{T}_{12},\mathcal{T}_{21},\mathcal{T}_{22}$. In particular, when $q = \frac{d_2}{d_2-2}$, we choose $\epsilon <\frac{d_2-d_1}{q'}$ and $\epsilon < 2-d_1$ for the esitmates of $\mathcal{T}_{12}$ and $\mathcal{T}_{22}$ respectively.

\medskip

$\bullet$ \textit{Case 4. $2<d_1<d_2$}. We have for any $\epsilon>0$,
\begin{align*}
    |\mathcal{T}_{ij}g(x)| \lesssim_{\epsilon} \|g\|_{q'} \begin{cases}
        \mathcal{X}_1(x) \begin{cases}
            |x|^{2-d_1-d_1/q'}, & 1<q<\frac{d_1}{d_1-2},\\
            |x|^{2-d_1-d_1/q'+\epsilon}, & q = \frac{d_1}{d_1-2},\\
            |x|^{-d_1}, & q> \frac{d_1}{d_1-2},
        \end{cases}, & i=1, j=1,\\
        \mathcal{X}_1(x) \begin{cases}
            |x|^{2-d_1-d_2/q'}, & 1<q<\frac{d_2}{d_2-2},\\
            |x|^{2-d_1-d_2/q'+\epsilon}, & q = \frac{d_2}{d_2-2},\\
            |x|^{-d_1}, & q> \frac{d_2}{d_2-2},
        \end{cases}, & i=1, j=2,\\
        \mathcal{X}_2(x) \begin{cases}
            |x|^{2-d_2-d_1/q'}, & 1<q<\frac{d_1}{d_1-2},\\
            |x|^{2-d_2-d_1/q'+\epsilon}, & q = \frac{d_1}{d_1-2},\\
            |x|^{-d_2}, & q> \frac{d_1}{d_1-2},
        \end{cases}, & i=2, j=1,\\
        \mathcal{X}_2(x) \begin{cases}
            |x|^{2-d_2-d_2/q'}, & 1<q<\frac{d_2}{d_2-2},\\
            |x|^{2-d_2-d_2/q'+\epsilon}, & q = \frac{d_2}{d_2-2},\\
            |x|^{-d_2}, & q> \frac{d_2}{d_2-2},
        \end{cases}, & i=2, j=2.
    \end{cases}
\end{align*}
The proof follows by applying Hölder's inequality to all situations and we omit details.

\end{proof}

Recall bilinear form:
\begin{align*}
    \mathcal{B}(f,g) &= - \sum_{i,j=1}^2 \int_{\Tilde{\mathbb{R}}} f(x) \int_0^1 \lambda^2 k_i(\lambda |x|) \int_{\Tilde{\mathbb{R}}} k_j(\lambda|y|)g(y)d\mu(y) F(\lambda) d\lambda d\mu(x)\\
    &= - \sum_{i,j=1}^2 \int_{\Tilde{\mathbb{R}}} \frac{f(x)}{|x|} \mathcal{T}_{ij}g(x) d\mu(x).
\end{align*}
By Hölder's inequality, Lemma~\ref{le_key} and Lemma~\ref{le_Hardy}, we conclude for $1<p<d_*$,
\begin{align*}
    |\mathcal{B}(f,g)|\le \sum_{i,j=1}^2 \left\| \frac{f}{|\cdot|}\right\|_p \|\mathcal{T}_{ij}g\|_{p'} \lesssim \|\nabla f\|_p \|g\|_{p'}.
\end{align*}
By ranging $g\in C_c^\infty(\Tilde{\mathbb{R}})$ with $\|g\|_{p'}=1$, we finally deduce for all $p$ in \eqref{range},
\begin{align*}
    \|\Delta^{1/2}f\|_p = \sup_{\|g\|_{p'}=1} \left| \left\langle \Delta^{1/2}f, g \right\rangle \right| = \sup_{\|g\|_{p'}=1} \left| \left\langle \nabla f, R_{kl}g \right\rangle + \left\langle \nabla f, T_{h}g \right\rangle + \mathcal{B}(f,g) \right|\lesssim \|\nabla f\|_p 
\end{align*}
as desired.

The proof of Theorem~\ref{thm_RR_main} is now complete.
\end{subsection}

\appendix
\section{Estimates for Low Energy Parts}\label{appendix}
We summarize the estimates for each quadrant in the following tables. 

\textbf{If} $1<d_1<d_2<2$.
\begin{table}[H]
    \centering
    \begin{tabular}{|c|c|c|c|c|}
    \hline
    \diagbox{If}{Qua}     & $I_1$ & $II_1$ & $III_1$ & $IV_1$   \\
    \hline
    $|x|\le |y|$ & $|x|^{1-d_2}|y|^{d_1-d_2-1}$     & $|x|^{1-d_1}|y|^{d_1-d_2-1}$ & $|x|^{1-d_1}|y|^{-1}$
     & $|x|^{1-d_2}|y|^{-1}$ 
    \\
    \hline
    $|x|\ge |y|$ &  $|x|^{d_1-2d_2}$ & $|x|^{-d_2}$  &  $|x|^{-d_1}$
     & $|x|^{-d_2}$ 
     \\
    \hline
    \end{tabular}
    \label{tab2}
\end{table}

$\textbf{If}$ $1<d_1<d_2=2$. Then for any $\epsilon>0$
\begin{table}[H]
    \centering
    \begin{tabular}{|c|c|c|c|c|}
    \hline
    \diagbox{If}{Qua}     & $I_1$ & $II_1$ & $III_1$ & $IV_1$   \\
    \hline
    $|x|\le |y|$ & $|x|^{-1}|y|^{d_1-3}$     & $|x|^{1-d_1}|y|^{d_1-3}$ & $|x|^{1-d_1}|y|^{-1}$
     & $|x|^{-1}|y|^{-1}$ 
    \\
    \hline
    $|x|\ge |y|$ &  $|x|^{d_1-4+\epsilon}|y|^{-\epsilon}$ & $|x|^{-2+\epsilon}|y|^{-\epsilon}$  &  $|x|^{-d_1}$
     & $|x|^{-2}$ 
     \\
    \hline
    \end{tabular}
    \label{tab3}
\end{table}

$\textbf{If}$ $1<d_1<2<d_2$.
\begin{table}[H]
    \centering
    \begin{tabular}{|c|c|c|c|c|}
    \hline
    \diagbox{If}{Qua}    & $I_1$ & $II_1$ & $III_1$ & $IV_1$   \\
    \hline
    $|x|\le |y|$ & $|x|^{1-d_2}|y|^{d_1-d_2-1}$     & $|x|^{1-d_1}|y|^{d_1-d_2-1}$ & $|x|^{1-d_1}|y|^{-1}$
     & $|x|^{1-d_2}|y|^{-1}$ 
    \\
    \hline
    $|x|\ge |y|$ &  $|x|^{d_1-d_2-2}|y|^{2-d_2}$ & $|x|^{-2}|y|^{2-d_2}$  &  $|x|^{-d_1}$
     & $|x|^{-d_2}$ 
     \\
    \hline
    \end{tabular}
    \label{tab4}
\end{table}

$\textbf{If}$ $2<d_1<d_2$.
\begin{table}[H]
    \centering
    \begin{tabular}{|c|c|c|c|c|}
    \hline
    \diagbox{If}{Qua}     & $I_1$ & $II_1$ & $III_1$ & $IV_1$   \\
    \hline
    $|x|\le |y|$ & $|x|^{1-d_2}|y|^{1-d_2}$     & $|x|^{1-d_1}|y|^{1-d_2}$ & $|x|^{1-d_2}|y|^{1-d_1}$
     & $|x|^{1-d_1}|y|^{1-d_1}$ 
    \\
    \hline
    $|x|\ge |y|$ &  $|x|^{-d_2}|y|^{2-d_2}$ & $|x|^{-d_1}|y|^{2-d_2}$  &  $|x|^{-d_2}|y|^{2-d_1}$
     & $|x|^{-d_1}|y|^{2-d_1}$ 
     \\
    \hline
    \end{tabular}
    \label{tab5}
\end{table}

{\bf Acknowledgments.} 
The author wants to thank his Ph.D. supervisor Adam Sikora and Professor Andrew Hassell for introducing him to the topic and giving valuable suggestions.

\bibliographystyle{abbrv}
\bibliography{name}

\end{document}